\date{14 September 2014}

\documentclass[11pt,reqno]{amsart}
\usepackage{latexsym,amsmath,amsfonts,amscd,amssymb}
\usepackage{graphics,color}

\theoremstyle{plain}
\newtheorem{theorem}{Theorem}[section]
\newtheorem{corollary}[theorem]{Corollary}
\newtheorem{lemma}[theorem]{Lemma}
\newtheorem{proposition}[theorem]{Proposition}
\theoremstyle{definition}

\newtheorem{remark}[theorem]{Remark}

\newcommand{\s}{\sigma}

\newcommand{\G}{\Gamma}

\newcommand{\cM}{\mathcal{M}}

\newcommand{\cU}{\mathcal{U}}
\newcommand{\cV}{\mathcal{V}}
\newcommand{\cX}{\mathcal{X}}

\newcommand{\LL}{\mathbb{L}}
\newcommand{\RR}{\mathbb{R}}
\newcommand{\CC}{\mathbb{C}}

\newcommand{\PP}{\mathbb{P}}

\newcommand{\ZZ}{\mathbb{Z}}

\newcommand{\inc}{\hookrightarrow}
\newcommand{\imat}{\sqrt{-1}}

\newcommand{\too}{\longrightarrow}
\newcommand{\x}{\times}

\newcommand{\la}{\langle}
\newcommand{\ra}{\rangle}

\newcommand{\frS}{{\frak S}}

\newcommand{\Id}{\mathrm{Id}}
\newcommand{\id}{\mathrm{Id}}

\newcommand{\Sym}{\mathrm{Sym}}
\newcommand{\VarC}{\cV ar_\CC}

\DeclareMathOperator{\Hom}{Hom}
 \DeclareMathOperator{\diag}{diag}
 
\DeclareMathOperator{\Gr}{Gr}
\DeclareMathOperator{\GL}{GL}
\DeclareMathOperator{\PGL}{PGL}
\DeclareMathOperator{\SL}{SL}

\DeclareMathOperator{\Stab}{Stab}
\DeclareMathOperator{\Spec}{Spec}
\DeclareMathOperator{\tr}{tr}

\title[$\SL(3,\CC)$-character variety of torus knots]{Geometry of the $\SL(3,\CC)$-character variety \\ of torus knots}


\author[V. Mu\~{n}oz]{Vicente Mu\~{n}oz}
\address{Facultad de Matem\'aticas, Universidad Complutense de Madrid,
Plaza Ciencias 3, 28040 Madrid, Spain}
\address{Instituto de Ciencias Matem\'aticas (CSIC-UAM-UC3M-UCM),
C/ Nicol\'as Cabrera 15, 28049 Madrid, Spain}
\email{vicente.munoz@mat.ucm.es}

\author[J. Porti]{Joan Porti}
\address{Departament de Matem\`atiques. Universitat Aut\`onoma de Barcelona, 08193 Bellaterra. Spain}
\email{porti@mat.uab.cat}

\begin{document}

\begin{abstract}
  Let $G$ be the fundamental group of the complement of the
  torus knot of type $(m,n)$. This has a presentation $G=\la x,y \, |
  \, x^m=y^n\ra$. We find the geometric description of the character
  variety $X(G)$ of characters of representations of $G$ into $\SL(3,\CC)$, $\GL(3,\CC)$ and $\PGL(3,\CC)$.
\end{abstract}

\maketitle

\section{Introduction}\label{sec:introduction}

Since the foundational work of Culler and Shalen \cite{CS}, 
the varieties of $\SL(2,\CC)$-characters have been extensively studied. 
Given a manifold $M$, the variety of representations of
$\pi_1(M)$ into $\SL(2,\CC)$ and the variety of characters of such representations both
contain information of the topology of $M$. This is specially
interesting for $3$-dimensional manifolds, where the fundamental
group and the geometrical properties of the manifold are 
strongly related. This can be used to study knots $K\subset S^3$, by analysing the
$\SL(2,\CC)$-character variety of the fundamental group of the knot complement
$S^3-K$ (these are called \emph{knot groups}). 

For a very different reason, the case of fundamental groups of 
surfaces has also been extensively analysed \cite{hausel-thaddeus, hitchin, lomune}, 
in this situation focusing more on geometrical properties of the moduli
space in itself (cf.\ non-abelian Hodge theory).

However, much less is known of the character varieties for other groups, notably
for $\SL(r,\CC)$ with $r\geq 3$. The character varieties for $\SL(3,\CC)$ for
free groups have been described in \cite{Lawton, LM}. In the case of $3$-manifolds,
little has been done. In this paper, we study the case of the torus knots $K_{m,n}$ 
of any type $(m,n)$, which are the first family of knots where the computations are
rather feasible. The case of $\SL(2, \CC)$-character varieties of torus knots was carried
out in  \cite{Martin-Oller, Munoz}. For $\SL(3,\CC)$, the torus knot $K_{2,3}$ has
been done in \cite{HeusenerP}. 

In the case of $\SL(2,\CC)$-character varieties of torus knot groups, only one-dimensional 
irreducible components appear. However, when we move to $\SL(3,\CC)$, we see components of 
different dimensions. In the case of torus knots, we shall see components of dimension $4$ and
of dimension $2$. Our main result is an explicit geometrical description of the $\SL(3,\CC)$-character
variety of torus knots.

\begin{theorem}\label{thm:main}
Let $m,n$ be coprime positive integers. By swapping them if necessary, assume that $n$ is odd.
The $\SL(3,\CC)$-character variety $\cX_3$ of the torus knot $K_{m,n}\subset S^3$ is stratified into
the following components:
 \begin{itemize}
 \item One component consisting of totally reducible representations, isomorphic to $\CC^2$.
 \item $[\frac{n-1}2][\frac{m-1}2]$ components consisting of  partially reducible representations, each isomorphic to
 $(\CC-\{0,1\})\times \CC^*$. 
 \item If $n$ is even, there are $(m-1)/2$ extra components consisting of partially reducible representations, each 
 isomorphic to $\{(u,v)\in \CC^2 | v\neq 0, v\neq u^2\}$. 
 \item  $\frac1{12}(n-1)(n-2)(m-1)(m-2)$ componens of dimension $4$, consisting of irreducible
representations, all isomorphic to each other, and which are described explicitly in Remark \ref{rem:otro-mas}. 
\item $\frac12 (n-1)(m-1)(n+m-4)$ components consisting of irreducible representations, 
each isomorphic to $(\CC^*)^2-\{x+y=1\}$.
\end{itemize}
Moreover $m,n$ can be recovered from the above information.
\end{theorem}

We also describe geometrically how these components fit into the whole of the $\SL(3,\CC)$-character variety,
that is what is the closure of each of the strata in $\cX_3$. In particular, the total space is connected.
In Section \ref{sec:GL3} we give also the corresponding result for  the $\GL(3,\CC)$-character variety
and the $\PGL(3,\CC)$-character variety of the torus knots. 
Finally, we compute the K-theory class (in the Grothendieck ring of varieties) of all these character varieties.

\noindent \textbf{Acknowledgements.} We are grateful to Marina Logares, Jorge Mart\'{\i}n-Morales, Maite Lozano,
Javier Mart\'{\i}nez and Sean Lawton for
their interest in this work. The contribution of Henrik R\"uping in Mathoverflow (mathoverflow.net/questions/18225) has
been useful in our argument in  Proposition \ref{prop:quotient}. The second author is partially supported by Micinn through
grant MTM2012-34834.

\section{Moduli of representations and character varieties}\label{sec:character}

Let $\G$ be a finitely presented group, and let $G=\SL(r,\CC)$, $\GL(r,\CC)$ or $\PGL(r,\CC)$. 
A \textit{representation} of $\G$ in $G$ is a homomorphism $\rho: \G\to G$.
Consider a presentation $\G=\la x_1,\ldots, x_k | r_1,\ldots, r_s \ra$. Then $\rho$ is completely
determined by the $k$-tuple $(A_1,\ldots, A_k)=(\rho(x_1),\ldots, \rho(x_k))$
subject to the relations $r_j(A_1,\ldots, A_k)=\id$, $1\leq j \leq s$. The space
of representations is
 \begin{eqnarray*}
 R(\G,G) &=& \Hom(\G, G) \\
  &=& \{(A_1,\ldots, A_k) \in G^k \, | \,
 r_j(A_1,\ldots, A_k)=\id, \,  1\leq j \leq s \}\subset G^{k}\, .
 \end{eqnarray*}
Therefore $R(\G,G)$ is an affine algebraic set.

We say that two representations $\rho$ and $\rho'$ are
equivalent if there exists $P\in G$ such that $\rho'(g)=P^{-1} \rho(g) P$,
for every $g\in G$. This corresponds to a change of basis in $\CC^r$.
Note that the action of $G$ descends to an action of the
projective group $PG=\PGL(r,\CC)$.
This produces an action of $\PGL(r,\CC)$ in $R(\G,G)$. The moduli space of representations
is the GIT quotient
 $$
 M(\G,G) = R(\G,G) // G \, .
 $$
Recall that by definition of GIT quotient for an affine variety, if we write
$ R(\G,G)=\Spec A$, then $M (\G,G)=\Spec A^{G}$.

A representation $\rho$ is \textit{reducible} if there exists some proper 
subspace $V\subset \CC^r$ such that for  all $g\in G$ we have 
$\rho(g)(V)\subset V$;  otherwise $\rho$ is
\textit{irreducible}. Note that if $\rho$ is reducible, then let  $V\subset \CC^r$
an invariant subspace, and consider a complement $\CC^r =V\oplus W$. 
Let $\rho_1=\rho|_V$ and $\rho_2$ the induced representation on the
quotient space $W=\CC^r/V$. Then we can write $\rho=\begin{pmatrix} \rho_1 & 0\\
f& \rho_2\end{pmatrix}$, where $f: \G \to \Hom(W,V)$. Take $P_t=
\begin{pmatrix} t^{r-k}\Id & 0\\ 0& t^{-k}\Id \end{pmatrix}$, where $k=\dim V$.
Then $P_t^{-1}\rho P_t=\begin{pmatrix} \rho_1 & 0\\
t^{r} f& \rho_2\end{pmatrix} \to \rho'=\begin{pmatrix} \rho_1 & 0\\
0& \rho_2\end{pmatrix}$, when $t\to 0$. Therefore $\rho$ and $\rho'$
define the same point in the quotient $M(\G,G)$. Repeating this, we can
substitute any representation $\rho$ by some $\tilde\rho=\bigoplus \rho_i$,
where all $\rho_i$ are irreducible representations. We call this process 
\emph{semi-simplification}, and $\tilde\rho$ a semisimple
representation; also $\rho$ and $\tilde\rho$ are called
S-equivalent. The space $M(\G,G)$ parametrizes semisimple representations
\cite[Thm.~ 1.28]{LuMa}.

Suppose now that $G=\SL(r,\CC)$.
Given a representation $\rho: \G\to G$, we define its
\textit{character} as the map $\chi_\rho: \G\to \CC$,
$\chi_\rho(g)=\tr \rho (g)$. Note that two equivalent
representations $\rho$ and $\rho'$ have the same character.
There is a character map $\chi: R(\G,G)\to \CC^\G$, $\rho\mapsto
\chi_\rho$, whose image
 $$
 X(\G,G)=\chi(R(\G,G))
 $$
is called the \textit{character variety of $\G$}. Let us give
$X(\G,G)$ the structure of an algebraic variety. The traces $\chi_\rho$
span a subring $B\subset A$. Clearly $B\subset A^{G}$. As
$A$ is noetherian, we have that $B$ is a finitely generated $\CC$-algebra. Hence
there exists a collection $g_1,\ldots, g_a$ of elements of
$G$ such that $\chi_\rho$ is determined by $\chi_\rho(g_1),\ldots,
\chi_\rho(g_a)$, for any $\rho$. Such collection gives a map
 $$
  \Psi:R(\G,G)\to \CC^a\, , \qquad
  \Psi(\rho)=(\chi_\rho(g_1),\ldots, \chi_\rho(g_a))\, ,
 $$
and $X(\G,G)\cong \Psi(R(\G,G))$. This endows $X(\G,G)$ with
the structure of an algebraic variety, which is independent of the chosen collection.
The natural algebraic map  
 $$
 M(\G,G)\to X(\G,G)
 $$ 
is an  isomorphism (see Chapter 1 in \cite{LM}). This is the same as to say that 
$B=A^G$, that is, the ring of invariant polynomials is generated by characters.

\section{Grothendieck ring of varieties} \label{sec:KVar}

Let $\VarC$ be the category of quasi-projective complex varieties.
We denote by $K (\VarC)$ the \emph{Grothendieck ring} of
$\VarC$. This is the abelian group generated by elements $[Z]$, for
$Z \in \VarC$, subject to the relation $[Z]=[Z_1]+[Z_2]$ whenever $Z$
can be decomposed as a disjoint union $Z=Z_1\sqcup Z_2$ of a closed and
a Zariski open subset. 

There is a naturally defined product in $K (\VarC)$ given by $[Y]\cdot
[Z]=[Y\x Z]$. Note that if $\pi:Z\to Y$ is an algebraic fiber
bundle with fiber $F$, which is locally trivial in the Zariski
topology, then $[Z]=[F]\cdot [Y]$.

We denote by $\LL=[\CC]$ the \emph{Lefschetz object} in $K (\VarC)$. Clearly $\LL^k=[\CC^k]$. 
The following result will be useful later on.

\begin{proposition} \label{prop:quotient}
Let $\mu_r$ act on $X=(\CC^*)^k$ by $(t_1,\ldots,t_k) \mapsto (\xi^{a_1}t_1,\ldots, \xi^{a_k}t_k)$,
for some weights $a_1,\ldots,a_k\in \ZZ$, $\xi=e^{2\pi \sqrt{-1}/r}$. Then $X/\mu_r \cong (\CC^*)^k$.

As a consequence, for the same action on $Y=\CC^k$, we have $[Y/\mu_r]=\LL^r=[Y]$.
\end{proposition}

\begin{proof}
If the action is not free, then it factors through some quotient of $\mu_r$, and we substitute it by the
action of the quotient group. So we can assume that the action is free. 
If $k=1$ then the result is trivially true, since the quotient $\CC^*/\mu_r$ is parametrized by
$w=t^r\in \CC^*$.

Suppose $k>1$. There are integer
numbers $b_1,\ldots,b_k,b$ such that $a_1b_1+\ldots+a_kb_k+rb=1$. Consider the quoient
 $$ 
 \ZZ_r^k \to \ZZ_r
 $$
given by $(x_1,\ldots, x_k) \mapsto a_1x_1+\ldots +a_kx_k$. This is surjective, and the choice
of $(b_1,\ldots,b_k)$ gives a splitting. Let $K$ be the kernel. Then $\ZZ_r^k\cong K \oplus \ZZ_r$.
Clearly, it must be $K\cong \ZZ_r^{k-1}$. The inclusion $K\inc \ZZ_r^k$ defines a collection 
of $k-1$ vectors $(c_{i1},\ldots, c_{ik})$, $i=1, \ldots, k-1$. The matrix 
 $$
 M=\begin{pmatrix} c_{11} & \ldots & c_{1k} \\ \vdots & & \vdots \\ c_{k-1,1} &\ldots & c_{k-1,k} \\
 b_1 & \ldots & b_k \end{pmatrix}
 $$
has determinant $d=\det(M)$, which is a unit in $\ZZ_r$. We want to modify the entries of the matrix
$M$ adding multiples of $r$ so that we get $\det(M)=1$, and so $M$ is invertible over $\ZZ$. First,
multiply the first row by an integer $x$ so that $\det(M) x\equiv 1 \pmod r$. So $M$ is still invertible
modulo $r$ and $\det(M)\equiv 1 \pmod r$. We can write $M=ABC$, where $A,C$ are invertible
over the integers, and $B=\diag(e_1,\ldots, e_k)$, $e_1|e_2|\ldots |e_k$. Note that $\prod e_i\equiv 1\pmod r$.
If we add multiples of $r$ to the entries of $B$, we do the same to the entries of $M$. So without loss
of generality, we can work with $M=B$. 

If all $e_i=1$, we have finished. Moreover, if all $e_i \equiv 1 \pmod r$, we can change each $e_i$ by
adding multiples of $r$ to arrange $e_i=1$. So now
suppose that there is one entry $e_j\neq 1 \pmod r$. As $\prod e_i\equiv 1\pmod r$, there must be another entry $e_l \neq 1
\pmod r$. Adding $r$ to $e_j$, we have that $e_j+r$ and $e_l$ are coprime (recall that $e_j|e_l$ or $e_l|e_j$, and all 
$e_i$ are coprime to $r$). Then we can diagonalize this new $M$ again, getting elemetary divisors $1$, $(e_j+r)e_l$ instead
of $e_j,e_l$. This increases the number of elements of the diagonal of $M$ equal to $1$. Repeating the
process, we can finally get all diagonal entries equal to $1$.

Now consider $u_j=t_1^{c_{j1}}\ldots t_k^{c_{jk}}$, $j=1,\ldots, k-1$, and $u_k=t_1^{b_1}\ldots t_k^{b_k}$. 
As $\det M=1$ this is a change of variables, so $(u_1,\ldots, u_k)$ parametrizes $(\CC^*)^k$, and
the action of $\mu_r$ is given by  $(u_1,\ldots, u_k) \mapsto (u_1,\ldots, \xi u_k)$. Therefore
 $X/\mu_r \cong (\CC^*)^k/\mu_r =(\CC^*)^{k-1} \x (\CC^*/\mu_r)\cong (\CC^*)^k$.

The last assertion follows by stratifying $Y$ according to how many entries are zero.
\end{proof}

\section{Character varieties of torus knots} \label{sec:torus}

Let $T^2=S^1 \times S^1$ be the $2$-torus and consider the standard embedding
$T^2\subset S^3$. Let $m,n$ be a pair of coprime positive integers. Identifying
$T^2$ with the quotient $\RR^2/\ZZ^2$, the image of the straight line $y=\frac{m}{n}
x$ in $T^2$ defines the \textit{torus knot} of type $(m,n)$, which we shall denote
as $K_{m,n}\subset S^3$ (see \cite[Chapter 3]{Rolfsen}).

For any knot $K\subset S^3$, we denote by $\G_K$ the fundamental group of the exterior
$S^3-K$ of the knot. It is known that
 $$
  \G_{m,n}= \G_{K_{m,n}} \cong \la x,y \, | \, x^n= y^m \,\ra \,.
 $$
The purpose of this paper is to describe the character variety
$X(\G_{m,n},G)$, for $G=\SL(r, \CC), \GL(r,\CC),\PGL(r,\CC)$.

We introduce the following notation:
 \begin{align*}
 \cX_r &= X(\G_{m,n}, \SL(r,\CC)), \\
 \tilde\cX_r &= X(\G_{m,n}, \GL(r,\CC)), \\
 \bar\cX_r &= X(\G_{m,n}, \PGL(r,\CC)),
 \end{align*}
dropping the reference to $m,n$ in the notation.

\begin{lemma}
\label{lemma:X1}
 $\tilde\cX_1\cong \CC^*$.
\end{lemma}

\begin{proof}
Let $(\lambda,\mu)\in \tilde\cX_1$. Then $\lambda^n=\mu^m$,
so there exists a unique $t\in \CC^*$ such that $\lambda=t^m,\mu=t^n$ (here we
use that $m,n$ are coprime).
This means that $\tilde\cX_1\cong\CC^*$ via $t\mapsto (t^m,t^n)$.
\end{proof}

There is a map $\det: \tilde\cX_r \to \tilde\cX_1$ given by $(\det \rho)(g)=\det (\rho(g))$, 
for any $g\in\G$. Then $\cX_r=\det^{-1}(\mathbf{1})$, where $\mathbf{1}$ is the
trivial character.
Otherwise said, if $\rho=(A,B)\in \tilde\cX_r$ then $(\det A, \det B)\in \tilde\cX_1$.
Here $\det A=t^m, \det B=t^n$, for some $t\in\CC^*$. We shall write
$\det \rho=t$. So $(A,B)\in\cX_r$ when $t=1$.

There is an action of $\CC^*$ on $\tilde\cX_r$ given by
 $$
 \gamma\cdot (A,B) =(\gamma^m A,\gamma^n B).
 $$ 
Note that $\det(\gamma\cdot\rho)=\gamma^r \,\det(\rho)$. The kernel of 
$\CC^* \x \cX_r\to \tilde\cX_r$ is given by the $r$-th roots of unity $\mu_r$. So
there is an isomorphism
 \begin{equation}\label{eqn:1}
  \tilde\cX_r  \cong (\cX_r \x \CC^*)/ \mu_r .
 \end{equation}

Now let $([A],[B])\in \bar\cX_r$, that is, $[A],[B]\in \PGL(r,\CC)$ with $[A^m]=[B^n]$.
There is a surjective map $\SL(r,\CC)\to \PGL(r,\CC)$ with kernel $\mu_r$. So we 
can assume $(A,B)\in \SL(r,\CC)$, and $A^n=\lambda B^m$ for some $\lambda\in \CC^r$.
Take determinants, we have $\lambda^r=1$. The matrices $A,B$ are well-defined
up to multiplication by $r$-th roots of unity. Let $\epsilon, \varepsilon \in \mu_r$. Then
$(\epsilon A)^n=\lambda' (\varepsilon B)^m$ with $\lambda'=\lambda \epsilon^n \varepsilon^{-m}$.
As $m,n$ are coprime, we can arrange $\lambda'=1$ by choosing $\epsilon,\varepsilon$ suitably. 
Also, if $\lambda=1$ then $\lambda'=1$ means that $\epsilon=t^m$, $\varepsilon=t^n$ with $t\in\mu_r$.
That is, we have the isomorphism
 \begin{equation}\label{eqn:2}
  \bar\cX_r \cong \cX_r/ \mu_r\,.
 \end{equation}
Equivalently, 
  $$
  \bar\cX_r \cong \tilde\cX_r/\CC^* \, .
 $$
Note that in particular this means that the representations of $\Gamma_{m,n}$ in $\PGL(r,\CC)$
all lift to $\SL(r,\CC)$ (this is not true for all fundamental groups of $3$-manifolds).
 
In \cite{Martin-Oller} the character variety $\cX_2$ is described by
finding a set of equations satisfied by the traces of the matrices of the
images by the representation. In \cite{Munoz} the same variety 
$\cX_2$ is described by a geometric method based on the
study of eigenvectors and eigenvalues of the matrices. Here we shall
extend the latter to study the varieties $\cX_3$, $\tilde\cX_3$, $\bar\cX_3$.

\section{Stratification of the character variety} \label{sec:r}

We denote by $\pi=(r_1,\stackrel{(a_1)}{\ldots}, r_1, \ldots,r_s,\stackrel{(a_s)}{\ldots}, r_s)$ a partition of $r$, that
is $a_1r_1+\ldots+a_sr_s=r$, $r_1>\ldots >r_s>0$, $a_j\geq 1$. Let $\Pi_r$
be the set of all partitions of $r$. 
We decompose the character variety
 $$
 \tilde\cX_r=\bigsqcup_{\pi\in \Pi_r} \tilde X_{\pi}\, ,
 $$
into locally closed subvarieties, where $\tilde X_{\pi}$ corresponds to 
representations 
 \begin{equation}\label{eqn:rho}
 \rho=\bigoplus_{t=1}^s \bigoplus_{l=1}^{a_t}  \rho_{tl} , \qquad
 \rho_{tl}: \G \too \GL(r_t,\CC).
 \end{equation}
Also
 $$
 \cX_r=\bigsqcup_{\pi\in \Pi_r}  X_{\pi}\, ,
 $$
where $X_\pi=\tilde X_\pi  \cap \cX_r$, that is, consisting of those (\ref{eqn:rho}) with
$\prod_{t,l} \det \rho_{tl}(g) =1$.

For $\PGL(r,\CC)$-representations, we have 
 $$
 \bar\cX_r=\bigsqcup_{\pi\in \Pi_r}  \bar X_{\pi}\, ,
 $$
with $\bar X_{\pi}$ the image of $X_\pi$ under the projection $\cX_r\to \bar\cX_r$.

The irreducible representations correspond to $\pi_0=(r)$. We
denote $\cX_r^*=X_{\pi_0}$,  $\tilde\cX_r^*=\tilde X_{\pi_0}$ and 
$\bar\cX_r^*=\bar X_{\pi_0}$.

\begin{proposition} \label{prop:2}
We have $\tilde X_{\pi}= \prod_{t=1}^s \Sym^{a_t}  \tilde \cX_{r_t}^*$. \hfill $\Box$
\end{proposition}

The representations corresponding to $\pi_1=(1 ,\stackrel{(r)}\ldots, 1)$ are called
\emph{totally reducible}, since they have matrices that diagonalize simultaneously in the same basis. The 
corresponding sets will be denoted $X_{TR}=X_{\pi_1}$, $\tilde X_{TR}=\tilde X_{\pi_1}$ and 
$\bar X_{TR}=\bar X_{\pi_1}$.

\begin{proposition} \label{prop:3}
We have $\tilde X_{TR}\cong \CC^{r-1}\x \CC^*$ and $X_{TR}\cong \CC^{r-1}$. 

In particular, $[\tilde X_{TR}]=\LL^{r-1}(\LL-1)$, $[X_{TR}]=\LL^{r-1}$ and $[\bar X_{TR}]=\LL^{r-1}$.
\end{proposition}

\begin{proof}
By Lemma \ref{lemma:X1}, $\tilde\cX_1^* \cong \CC^*$ and it is formed by representations
$(t^m,t^n)$. By Proposition \ref{prop:2}, $\tilde X_{TR}=\Sym^r \tilde\cX_1^*$, where
$(A,B)$ is given by  $A=\diag (t_1^m,\ldots, t_r^m)$, $B=\diag (t_1^n,\ldots, t_r^n)$,  $t_j\in \CC^*$.
Then 
 $$
 \tilde X_{TR} \cong \Sym^r \CC^* \cong \CC^{r-1}\x \CC^*\, .
$$
Here the last isomorphism is given by $(t_1,\ldots, t_r)\mapsto (\s_1,\ldots,\s_r)$,
where $\s_k=\s_k(t_1, \ldots, t_r)$ is the $k$-th elementary symmetric function on
$t_1,\ldots, t_r$. Note that $t_i\neq 0, \forall i \iff \s_r\neq 0$.

The condition that $\det A=\det B=1$ means that $\prod t_j^m=\prod t_j^n=1$. So
$\prod t_j=1$, which is translated into $\s_r=1$. Hence $X_{TR}=\CC^{r-1}$.

For analysing the case of  $\bar X_{TR}$, we look at the action of $\mu_r$ on $X_{TR}$. Note that
$\epsilon\in \mu_r$ acts as $t_j\mapsto \epsilon t_j$, hence it acts as 
$(\s_1,\s_2,\ldots, \s_{r-1}) \mapsto  (\epsilon\s_1,\epsilon^2\s_2,\ldots, \epsilon^{r-1}\s_{r-1})$.
Proposition \ref{prop:quotient} gives now the result.
\end{proof}

\begin{lemma}\label{lem:red}
  Suppose that $\rho=(A,B)\in \tilde\cX_r^*$. Then $A,B$ are both diagonalizable
and $A^n=B^m=\varpi\, \Id$, for some $\varpi\in\CC^*$. Moreover,
neither $A$ nor $B$ is a multiple of the identity.

If $(A,B)\in\cX_r^*$ then $\varpi\in\mu_r$.
\end{lemma}

\begin{proof}
Choose a suitable basis so that $A$ is of Jordan form, with 
blocks $J_1,\ldots, J_k$. Let $J_i$ be a Jordan block of size $m_i\geq 1$ and
eigenvalue $\lambda_i$. Then $A^n$ has blocks $J_1^n,\ldots, J_k^n$,
and each $J_i^n$ is conjugated to a Jordan matrix of size $m_i$ with eigenvalue
$\lambda_i^n$. In particular, if $v_i$ is the eigenvector of $J_i$, then 
$v_i$ is the eigenvector of $J_i^n$. The span of eigenvectors of 
$A^n$ is $W=\la v_1,\ldots, v_k\ra$. This satisfies $A(W) \subset W$.
In an analogous fashion, as $A^n=B^m$, we have that $B(W)\subset W$.
By irreducibility, $W=\CC^r$, so all $m_i=1$. That is, both $A,B$ are diagonalizable.

Now let $\varpi=\lambda_1^n$. Let $W$ be the span of those $v_i$ such
that $\lambda_i^n=\varpi$. Then $A(W)\subset W$. In a similar fashion,
$B(W)\subset W$, so $W=V$. This means that $A^n=\varpi\, \Id$. Note that
all eigenvalues $\lambda_i$ of $A$ 
and all eigenvalues $\mu_j$ of $B$ satisfy $\lambda_i^n=\mu_j^m=\varpi$. 

The last assertion is clear since $\det A=\det B=1$ implies $\varpi^r=1$.
\end{proof}

\begin{corollary}\label{cor:5}
$\cX_r $, $\tilde\cX_r$ and $\bar\cX_r$ are connected.
\end{corollary}

\begin{proof}
By (\ref{eqn:1}) and (\ref{eqn:2}), it is enough to see that $\cX_r$ is connected. 
We just need to see that the closure of any component $X_\pi$ intersects $X_{TR}$, which 
is connected by Proposition \ref{prop:3}. 

Let us focus first on $\cX_r^*$. For $(A,B)\in \cX_r^*$,
we diagonalize $A$ and $B$ (by Lemma \ref{lem:red}). This gives decompositions
 $\CC^r=V_1\oplus \ldots \oplus V_s$
and $\CC^r=W_1\oplus  \ldots  \oplus W_l$  into eigenspaces
given for $A,B$, respectively. Let  $v_i=\dim V_i$, $w_j=\dim W_j$,
where $r=\sum v_i=\sum w_j$. By Lemma \ref{lem:red},
$A^n=B^m=\varpi\, \Id$, and $\varpi^r=1$, 
the eigenvalues $\epsilon_1, \ldots, \epsilon_s$ and
$\varepsilon_1,\ldots,\varepsilon_l$ for $A$ and $B$, respectively, satisfy 
 $\epsilon_i^n=\varepsilon_j^m=\varpi$, $\varpi^r=1$.  So there are finitely many
choices for $v_i,w_j,\epsilon_i,\varepsilon_j$. We denote $\kappa=
((\epsilon_1,\stackrel{(v_1)}{\ldots}, \epsilon_1, \ldots,\epsilon_s,\stackrel{(v_s)}{\ldots}, \epsilon_s),
(\varepsilon_1,\stackrel{(w_1)}{\ldots}, \varepsilon_1, \ldots,\varepsilon_l,\stackrel{(w_l)}{\ldots}, \varepsilon_l))$,
repeating eigenvalues according to multiplicity.
This gives a collection of (disjoint) components 
 \begin{equation}\label{eqn:kappa}
 \cX_{r}^*= \bigsqcup_\kappa \cX_{r,\kappa}^*.
 \end{equation}

Fix a component $\cX_{r,\kappa}^*$. To determine the pair $(A,B)$ it is enough to give 
the eigenspaces $V_1,\ldots, V_s$ and $W_1,\ldots, W_l$. These
are given by a point in the product $\prod \Gr(v_i,r) \x \prod \Gr(w_j,r)$. 
The set of possible points determining an irreducible representation is an open (if non-empty)
subset $\cU_\kappa\subset \prod \Gr(v_i,r) \x \prod \Gr(w_j,r)$, and
 \begin{equation}\label{eqn:kappa2}
 \cX_{r,\kappa}^* = \cU_\kappa/\PGL(r,\CC).
 \end{equation}
 This space is irreducible and hence
connected. The choice of the subspaces  $V_i=
\la e_{v_1+\ldots+v_{i-1}+1}, \ldots, e_{v_1+\ldots+v_i}\ra$,
$W_j=\la e_{w_1+\ldots+w_{j-1}+1}, \ldots, e_{w_1+\ldots+w_j}\ra$,
in terms of the standard basis $\{e_1,\ldots,e_r\}$ gives a representation in
the closure of  $\cX_{r,\kappa}^*$, which is totally reducible. This completes the
argument in this case.

Now consider another stratum $X_\pi$, for 
$\pi=(r_1,\stackrel{(a_1)}{\ldots}, r_1, \ldots,r_s,\stackrel{(a_s)}{\ldots}, r_s)$. 
By Proposition \ref{prop:2}, $\tilde X_{\pi}= \prod_{t=1}^s \Sym^{a_t}  \tilde \cX_{r_t}^*$. 
Take an irreducible component $Y_{tl}$ of $\cX_{r_t}^*$, for each $t=1,\ldots, s$, $l=1,\ldots, a_t$.
Let $\tilde Y_{tl}$ be the image of $\CC^*\x Y_{tl}$ in $\tilde\cX_{r_t}^*$.
Consider the map $\prod_{t,l} (\CC^*\x Y_{tl}) \to \tilde X_{\pi}$. The condition of the determinant
being $1$ gives a hypersurface $W\subset (\CC^*)^{a_1+\ldots+a_s}$, which is connected
since $a_1+\ldots+a_s>1$. The image of $W\x  \prod_{t,l}  Y_{tl}$ in $X_{\pi}$
is connected, and all of them cover $X_{\pi}$. Now the closure of $Y_{tl}$ contains elements
in $X_{TR}^{r_t}$. Hence the closure of the image of $W\x  \prod_{t,l}  Y_{tl}$
contains elements that are totally reducible, as required.
\end{proof}

\section{Maximal dimensional components} \label{sec:max-dim}

We will now count the number of maximal dimensional irreducible components of $\cX_r$, and give
a geometric description of them.

\begin{theorem} \label{thm:max}
The character variety $\cX_r$ has dimension at most $(r-1)^2$. For $r\geq 3$, the number of
irreducible components of this dimension is 
 $$
 \frac{1}{r} \binom{n-1}{r-1} \binom{m-1}{r-1} . 
 $$
In particular, there are no such components if either $n<r$ or $m<r$.
\end{theorem}

\begin{proof}
We work by induction $r$. The dimension holds for $r=1,2$. Take a 
stratum $X_\pi=  \prod_{t=1}^s \Sym^{a_t}  \tilde \cX_{r_t}^*$, then
 $$
 \dim X_\pi \leq \sum_{t=1}^s a_t(r_t-1)^2 < \left(\sum a_t r_t-1\right)^2=(r-1)^2,
 $$
unless $s=1$ and $a_1=1$, $r_1=r$, in which case there is equality. This corresponds to the component
$\cX_r^*$. 

As in the proof of Corollary \ref{cor:5}, the space $\cX_r^*$ consists of disjoint components 
$\cX_{r,\kappa}^*$, see (\ref{eqn:kappa}), and each component $\cX_{r,\kappa}^*$ is
of the form (\ref{eqn:kappa2}), where $\cU_\kappa\subset \prod \Gr(v_i,r) \x \prod \Gr(w_j,r)$.
Now 
 \begin{align*}
 \dim \, \cU_\kappa &= \sum v_i(r-v_i)+\sum w_j(r-w_j) \\ &< (\sum v_i)(r-1)+ (\sum w_j)(r-1)=2r(r-1),
 \end{align*}
unless all $v_i=w_j=1$, in which case there is equality. We have to quotient by $\PGL(r,\CC)$, which has dimension $r^2-1$,
hence $\dim  \cX_{r,\kappa}^* \leq 2r^2-r-(r^2-1)=r^2-2r-1=(r-1)^2$, with equality only if all $v_i=w_j=1$.

Now let us count the number of irreducible components of such dimension. It is the 
same as to count the number of $(\epsilon_1,\ldots, \epsilon_r)$, $(\varepsilon_1, \ldots,  \varepsilon_r)$,
subject to
 \begin{align*}
 &\epsilon_1\cdots \epsilon_r =1, \, \epsilon_i \text{ distinct},\\
 &\varepsilon_1\cdots \varepsilon_r =1,\, \varepsilon_j \text{ distinct}, \\
 & \epsilon_i^n =\varepsilon_j^m =\varpi,\\  & \varpi^r=1.
 \end{align*}
Denote by 
 $$
 N(k_1,k_2)=\# \{ 
(\epsilon_1,\ldots, \epsilon_r) \text{ distinct}, (\varepsilon_1, \ldots,  \varepsilon_r) \text{ distinct}  |
\epsilon_i^n =e^{2\pi\imat k_1/r}, \varepsilon_j^m =e^{2\pi\imat k_2/r} \},
 $$
for any pair of integers $k_1,k_2$. We have to compute the sum $T=\sum_{k=0}^{r-1} N(k,k)$.

As $m,n$ are coprime, we have that the modulo-reduction map $\ZZ_{nm} \to \ZZ_n\x \ZZ_m$ is a bijection.
Therefore
 \begin{equation*}
 \sum_{k=0}^{nm-1} N(k,k)= \sum_{k_1=0}^{n-1} \sum_{k_2=0}^{m-1} N(k_1,k_2).
 \end{equation*}
Now $N(k_1,k_2)=N_n(k_1)N_m(k_2)$, where
 \begin{align*}
 N_n(k_1) &=\# \{ (\epsilon_1,\ldots, \epsilon_r) \text{ distinct} \, |\, \epsilon_i^n =e^{2\pi\imat k_1/r} \}, \\
 N_m(k_2) &=\# \{ (\varepsilon_1, \ldots,  \varepsilon_r) \text{ distinct} \, |\, \varepsilon_j^m =e^{2\pi\imat k_2/r} \}.
 \end{align*}
Clearly $\sum_{k_1=0}^{n-1} N_n(k_1)= \frac{n!}{(n-r)!}$
and $\sum_{k_2=0}^{m-1} N_m(k_2)=\frac{m!}{(m-r)!}$. So
  \begin{equation}\label{eqn:sum}
 \sum_{k=0}^{nm-1} N(k,k)= \frac{n!}{(n-r)!} \, \frac{m!}{(m-r)!}.
 \end{equation}

Note that the sum (\ref{eqn:sum}) is the same if we start at any other integer, i.e.  $\sum_{k=a}^{nm-1+a} N(k,k)$ gives
the same value.  Therefore $\sum_{k=0}^{rnm-1} N(k,k)=r \sum_{k=0}^{nm-1} N(k,k)$. 
Also $N(k_1,k_2)= N(k_1+a r, k_2+b r)$, for any $a,b\in \ZZ$. Therefore
 $\sum_{k=0}^{rnm-1} N(k,k)=nm\sum_{k=0}^{r-1} N(k,k)=nmT$. 
Thus
 $T= \frac{r}{nm} \frac{n!}{(n-r)!} \frac{m!}{(m-r)!}$, and the number of irreducible components is
(taking into account the permutations of $\epsilon_i$ and of $\varepsilon_j$)
 $$
 \frac{1}{(r!)^2} T= \frac{r}{nm} \binom{n}{r}\binom{m}{r}= \frac{1}{r} \binom{n-1}{r-1} \binom{m-1}{r-1} .
 $$
\end{proof}

The proof of Theorem \ref{thm:max} gives us a way to label the maximal dimensional components of $\cX_r^*$.
Let 
$$
F'= \{ ((\epsilon_1,\ldots, \epsilon_r), (\varepsilon_1, \ldots,  \varepsilon_r)) 
\text{ distinct} | \prod\epsilon_i  =\prod \varepsilon_j =\omega, \omega^r=1 \}
$$
and define the set $F=F'/(\frS_r\x\frS_r)$, where the first $\frS_r$ acts by permutation on the components of
$(\epsilon_1,\ldots, \epsilon_r)$, and the second $\frS_r$ acts by permutation on the components of
$(\varepsilon_1,\ldots, \varepsilon_r)$. We shall denote a $\tau\in F$ as
$\tau=[(\epsilon_1,\ldots, \epsilon_r), (\varepsilon_1, \ldots,  \varepsilon_r)]$.

Now we describe geometrically the maximal dimensional components of $\cX_r$. For this, we introduce some notation. 
Consider $\GL(r,\CC)$. Let $T\cong (\CC^*)^r$ be the set of
diagonal matrices, and consider the action of $T\x T$, where the first $T$ acts on the left
and the second $T$ acts on the right on $\GL(r,\CC)$. The set
 $D=\{(\lambda \Id, \lambda^{-1}\Id)\}\subset T\x T$ acts trivially, so there is an
effective action of $T\x_D T$. 

\begin{proposition}\label{prop:Xtau}
 For $\tau\in F$, the $(r-1)^2$-dimensional component component $\cX_{r,\tau}^*$ is isomorphic to
 $$
\cX_{r,\tau}^* \cong \cM/(T\x_D T).
 $$
where $\cM \subset \GL(r,\CC)$ is the open subset of stable points for the $(T\x_D T)$-action.

In particular, all components for different $\tau\in F$, are isomorphic.
\end{proposition}

\begin{proof}
A maximal dimensional component $\cX_{r,\tau}^*$ is
determined by $\tau=[(\epsilon_1,\ldots, \epsilon_r), (\varepsilon_1, \ldots,  \varepsilon_r)]\in F$.
We fix a lift to $F'$, that is we fix an order of the eigenvalues throughout, say
$((\epsilon_1,\ldots, \epsilon_r), (\varepsilon_1, \ldots,  \varepsilon_r))$.

Given a pair $(A,B)\in \cX_{r,\tau}^*$, recall that $A,B$ are diagonalizable with the prescribed eigenvalues.
Let $v_1,\ldots,v_r$ be the eigenvectors of $A$ and $w_1,\ldots,w_r$ be the eigenvectors of $B$. These
are well-defined up to scalar multiples. We use $v_1,\ldots,v_r$ as a basis for $\CC^r$, and write
$w_j=(a_{1j},\ldots, a_{rj})$ in these coordinates. This produces a matrix
 \begin{equation}\label{eqn:M}
 M= \begin{pmatrix} a_{11} &\ldots & a_{1r} \\
 \vdots & \ddots & \vdots \\
a_{r1} & \ldots & a_{rr} \end{pmatrix}.
 \end{equation}
Note that $\det(M)\neq 0$, since the vectors $w_1,\ldots,w_r$ are linearly independent.
Let $\cM$ be the set of those $M$ which yield irreducible representations $(A,B)$. This
is equivalent to the fact that there do not exist sub-collections $v_{b_1},\ldots v_{b_p}$ and
$w_{a_1},\ldots, w_{a_p}$,
$0< p <r$, such that $W=\langle w_{a_1},\ldots, w_{a_p} \rangle=\langle v_{b_1},\ldots v_{b_p}\rangle$,
since in this case such $W$ would be invariant. This condition translates into the fact that the sub-minor corresponding
to $\{a_1,\ldots, a_p\} \times(\{1, \ldots, r\}-\{b_1,\ldots, b_p\})$ is identically zero. 
This is equivalent to the condition that $M$ is not a stable point for the action of $T\x_D T$, that is, the orbit
of the point has another orbit in the closure or the action is non-free. Clearly, if $(A,B)$ is as above, then acting by 
$(\diag(x_i),\diag(y_j))\in T\x_D T$, where $x_i=\lambda$ for $i\in \{a_1,\ldots, a_p\}$,
$x_i=\lambda^{-1}$ for $i\not\in \{a_1,\ldots, a_p\}$, $y_j=\lambda$ for $j\not\in \{b_1,\ldots, b_p\}$,
and $y_j=\lambda^{-1}$ for $j\in \{b_1,\ldots, b_p\}$, and then taking $\lambda\to 0$, we get points
in a different orbit, or the action is not free (the complementary minor tends to zero).

Let $\cM \subset \GL(r,\CC)$ be the open subset of stable points for the $(T\x_D T)$-action.
Then a representation $(A,B)$ is determined by a matrix $M\in \cM$ modulo the possible rescaling of
basis vectors $v_1,\ldots,v_r$ (this corresponds to the action of $T$ on the left)
and of the eigenvectors $w_1,\ldots,w_r$ (this corresponds to the action of $T$ on the right).
Note that the irreducible component $\cM/(T\x_D T)$ has dimension $r^2-(2r-1)=(r-1)^2$, as expected.
\end{proof}

\begin{remark} \label{rem:x}
The closure of the stratum $\cX_{r,\tau}^*$ is obtained by adding semisimple reducible representations (which are
direct sums of irreducible representations of smaller rank). This correspond to adding matrices $M$
as in (\ref{eqn:M}) for which there is a sub-minor corresponding
to $\{a_1,\ldots, a_p\} \times(\{1, \ldots, r\}-\{b_1,\ldots, b_p\})$ identically zero, and at the same
time the sub-minor $(\{1, \ldots, r\}-\{a_1,\ldots, a_p\} )\times\{b_1,\ldots, b_p\}$ also vanish.
We are therefore looking at a polystable point (direct sum of stable points)
of the action of $T\x_D T$. This means that
the closure of  $\cX_{r,\tau}^*$ is isomorphic to the GIT quotient
  $$
 \GL(r,\CC)//(T\x_D T).
 $$
\end{remark}

Now we shall explain how to get the maximal dimensional components of $\tilde\cX_r$ and $\bar\cX_r$
although we are not going to do the explicit count of them for general $r$ (we shall do it later for $r=2,3$).

Let $\tau=[(\epsilon_1,\ldots, \epsilon_r), (\varepsilon_1, \ldots,  \varepsilon_r)] \in F$. For
$t\in \mu_r$, we have  $t\cdot \tau=[(t\epsilon_1,\ldots, t\epsilon_r), (t\varepsilon_1, \ldots, t\varepsilon_r)]$,
which gives another (or the same) component. The map $t:\cX_r \to \cX_r$ maps $t:\cX_{r,\tau}^* \to \cX_{r,t\cdot\tau}^*$.

The maximal dimensional components of $\bar\cX_r$ are parametrized by the coset space $F/\mu_r$. So the
number of them is the cardinality of $F/\mu_r$. For given $[\tau]\in F/\mu_r$, the corresponding component is
 \begin{equation}\label{eqn:SS}
 \bar \cX_{r,[\tau]}^*= \left(\bigsqcup_{ t \in \mu_r} \cX_{r,t\cdot \tau}^* \right)/\mu_r\, .
 \end{equation}
Let $S=\Stab(\tau)\subset \mu_r$. If $t\in S$, then  $t: \cX_{r,\tau}^* \to \cX_{r,\tau}^*$, and
also there is an element $(f_t,g_t) \in \frS_r\x \frS_r$ such that 
$((t\epsilon_1,\ldots, t\epsilon_r), (t\varepsilon_1, \ldots, t\varepsilon_r))=
(f_t(\epsilon_1,\ldots, \epsilon_r), g_t(\varepsilon_1, \ldots, \varepsilon_r))$. Using
the natural action of $\frS_r\x \frS_r$ on $\cM$ (multiplication on the right and on the left by permutation matrices),
we have that (\ref{eqn:SS}) is isomorphic to 
 $$
 \bar \cX_{r,[\tau]}^*= \cX_{r, \tau}^* /S = \cM/ (H\x S) .
 $$

The maximal dimensional components of $\tilde\cX_r$ are also  parametrized by the coset space $F/\mu_r$. They 
are of dimension $(r-1)^2+1$, and are isomorphic to
 $$
  \tilde \cX_{r,[\tau]}^* =(\CC^*\x \left(\bigsqcup_{ t \in \mu_r} \cX_{r,t\cdot \tau}^* \right) ) /\mu_r \cong
 (\CC^*\x \cX_{r, \tau}^* ) / S =(\CC^*\x \cM ) /(H\x S) .
 $$

\section{Character varieties for $\SL(2,\CC),\GL(2,\CC)$ and $\PGL(2,\CC)$} \label{sec:rank2}

From \cite{Munoz}, we have the following:

\begin{proposition}\label{prop:7.1}
\label{prop:cX2} The variety $\cX_2$ consists of the following irreducible components:
\begin{itemize}
 \item One component $X_{TR}\cong \CC$.
 \item $(n-1)(m-1)/2$ components forming the irreducible locus. Each of them is isomorphic to
 $\CC-\{0,1\}$. The closure of each component is $\CC$, and intersects $X_{TR}$ in two points.
\end{itemize}
\end{proposition}

\begin{proof}
 Here the only possible representations are either totally reducible or irreducible. The totally reducible locus
is given by Proposition \ref{prop:3}. The irreducible representations must all be of type $\cX_{2,\tau}^*$ for
some $\tau=((\epsilon_1,\epsilon_2),(\varepsilon_1,\varepsilon_2))$. The number of them is given in
Theorem  \ref{thm:max}, and it is $(n-1)(m-1)/2$. Each of them is parametrized by 
$\cM/H$, where $\cM$ is formed by the matrices $\begin{pmatrix} a & b\\c &d \end{pmatrix}$ with
all entries non-zero, and $H=T \x_D T$ ($T$ are the diagonal matrices, and they act on the right and
on the left on $\cM$). Using the action of $H$, we can arrange $a=b=c=1$. Hence the matrix
is determined by $d\in \CC-\{0,1\}$. However, we shall use the parameter $r=\frac{1}{1-d}\in \CC-\{0,1\}$.

The closure of $\cM/H$ is given by $\GL(2,\CC)//H$. By Remark \ref{rem:x}, we have to 
add matrices with two entries which are zero. If either $b=c=0$, then we have
the reducible representation $(\diag (\epsilon_1,\epsilon_2), \diag(\varepsilon_1,\varepsilon_2))$.
This corresponds to $r=0$. 
If $a=d=0$, then we have
the reducible representation $(\diag (\epsilon_1,\epsilon_2), \diag(\varepsilon_2,\varepsilon_1))$.
This corresponds to $r=1$. 
\end{proof}

Note that $[\cX_2]=\LL + \frac12 (n-1)(m-1) (\LL-2)$.

\begin{proposition}\label{prop:7.2}
\label{prop:barcx2}
The variety $\bar\cX_2$ consists of the following irreducible components:
\begin{itemize}
 \item One component $\bar X_{TR}\cong \CC$.
 \item $[\frac{n-1}2][\frac{m-1}2]$ components of the irreducible locus, each of them isomorphic to
 $\CC-\{0,1\}$. The closure of each component is $\CC$, and interesects $\bar X_{TR}$ in two points.
 \item if $n$ is even and $m$ is odd, $(m-1)/2$ components of the irreducible locus, each of them isomorphic to
 $\CC^*$. The closure of each component is $\CC$, and intersects $\bar X_{TR}$ in one point.
(The case $m$ even and $n$ odd is analogous.)
\end{itemize}
\end{proposition}

\begin{proof}
We use the description $\bar\cX_2\cong \cX_2/\mu_2$, where $\mu_2=\{\pm 1\}$.
 For the reducible component, we have $\bar X_{TR} =\CC/\mu_2\cong \CC$.

The irreducible components of $\cX_2$ are parametrized by the finite set $F$ of eigenvalues 
of $A$ and $B$. If $F'=\{(\epsilon_1,\epsilon_2),(\varepsilon_1,\varepsilon_2)$, 
$\epsilon_1\epsilon_2 =\varepsilon_1\varepsilon_2=1$, $\epsilon_i^n=\varepsilon_j^m= \pm 1\}$, then
$\mu_2$ acts on $F'$ by $((\epsilon_1,\epsilon_2),(\varepsilon_1,\varepsilon_2))\mapsto 
((-1) ^m\epsilon_1,(-1) ^m\epsilon_2),((-1) ^n\varepsilon_1,(-1)^n\varepsilon_2))$, and $F=F'/(\frS_2\x \frS_2)$. When both $m$ and $n$
are odd there are no fixed points, because fixed points occur when $\epsilon_2=-\epsilon_1$ and
$\varepsilon_2=-\varepsilon_1$, equivalently when $\epsilon_1^2=\varepsilon_1^2=-1$,  which contradicts 
$\epsilon_1^n=\varepsilon_1^m= \pm 1$. In this case $\# (F/\mu_2)=\# F/2 =\frac14 (n-1)(m-1) = \frac{n-1}2 \, \frac{m-1}2$.

Assume now that $n$ is even, then the same calculation shows that fixed points 
occur precisely when 
$\{\epsilon_1,\epsilon_2\}=\{\pm\sqrt{-1}\}$, for any admissible value of 
$\varepsilon_j$. This yields 
$(m-1)/2$ components. On those invariant components, the action of $\mu_2$ permutes 
two rows or two columns of the matrix
 $\begin{pmatrix} a & b\\c &d \end{pmatrix}$ in the proof of Proposition \ref{prop:7.1}. This 
amounts to map the parameter $d$ to $1/d$,
 equivalently, to map the parameter $r$ to $1-r$. So
the component is isomorphic to $(\CC-\{0,1\})/\mu_2 \cong \CC^*$.
\end{proof}

By Proposition \ref{prop:7.2}, we have that for $m,n$ odd,
$[\bar\cX_2]=\LL + \frac14 (n-1)(m-1) (\LL-2)$. For $n$ even and $m$ odd, we have
$[\bar\cX_2]=\LL + \frac14 (n-2)(m-1) (\LL-2)+\frac12 (m-1) (\LL-1)$. 

\begin{proposition} \label{prop:7.3}
The variety $\tilde\cX_2$ consists of the following 
irreducible components:
\begin{itemize}
 \item One component $\tilde X_{TR}\cong \CC\x\CC^*$.
 \item $[\frac{n-1}2][\frac{m-1}2]$ components of the irreducible locus, each of 
them is isomorphic to
 $(\CC-\{0,1\})\times \CC^*$. The closure of each component is $\CC\times \CC^*$, 
and intersects $\tilde X_{TR}$ in two $\CC^*$'s.
 \item If $n$ is even, there are $(m-1)/2$ extra components of the irreducible locus, each 
isomorphic to $\{(u,v)\in \CC^2 | v\neq 0, v\neq u^2\}$. The closure of each component is $\CC \times\CC^*$, and 
intersects $\tilde X_{TR}$ along a $\CC^*$.
(The case $m$ even and $n$ odd is analogous.)
\end{itemize}
\end{proposition}

\begin{proof}
The component $\tilde X_{TR}$ is given in  Proposition~\ref{prop:3}. 
Now we use the description $\tilde\cX_2\cong ( \cX_2\times\CC^*)/\mu_2$. By the 
proof of Proposition~\ref{prop:barcx2},
when $m$ and $n$ are odd $\mu_2$ switches components of $\cX_2$, without 
preserving any of them, and the proposition follows in this case.
When $n$ is even, $\mu_2$ preserves  $(m-1)/2$ components of $\cX_2$, and for 
each such component it
maps the parameters $(r,\lambda)\in (\CC-\{0,1\})\times \CC^*$ 
to $(1-r, -\lambda)$. The quotient is $((\CC-\{0,1\})\times \CC^*)/\mu_2$. The compactification
is given as $(\CC\times \CC^*)/\mu_2$, and can be coordinatized with $u=(2r-1)\lambda$, $v=\lambda^2$.
The quotient is given by $\{(u,v)\in \CC^2 |v\neq 0\}$, and the image of the curves $(\{0,1\} \times \CC^*)/\mu_2\cong
\CC^*$ is given by $v=u^2, u\neq 0$. Hence the component is isomorphic to
$\{(u,v)\in \CC^2 | v\neq 0, v\neq u^2\}$.
\end{proof}

As $\bar\cX_2=(\tilde\cX_2)/\CC^*$, we have $[\tilde\cX_2]=(\LL-1)[\bar\cX_2]$.
So for $m,n$ odd,
$[\tilde\cX_2]=(\LL + \frac14 (n-1)(m-1) (\LL-2))(\LL-1)$. For $n$ even and $m$ odd, we have
$[\tilde\cX_2]=(\LL + \frac14 (n-2)(m-1) (\LL-2)+\frac12 (m-1) (\LL-1))(\LL-1)$.

\section{Character varieties for $\SL(3,\CC)$} \label{sec:SL3}

Now we move to the description of the $\SL(3,\CC)$-character variety $\cX_3$. 

\begin{proposition}\label{prop:8.1}
The components of reducible representations of $\cX_3$ are the following:
 \begin{itemize}
 \item The component of totally reducible representations $X_{TR}=\CC^2$.
 \item $[\frac{n-1}2][\frac{m-1}2]$ components $X_{PR}^{1,i}$ of  partially reducible representations, each isomorphic to
 $(\CC-\{0,1\})\times \CC^*$. 
 \item If $n$ is even, there are $(m-1)/2$ extra components  $X_{PR}^{2,i}$ of partially reducible representations, each 
isomorphic to $\{(u,v)\in \CC^2 | v\neq 0, v\neq u^2\}$. 
(The case $m$ even and $n$ odd is analogous.)
\end{itemize}
\end{proposition}

\begin{proof}
The description of $X_{TR}$ is in Proposition \ref{prop:3}. 
Now we move to partially reducible representations. This corresponds to representations in $\SL(3,\CC)$ 
which split as $\CC^3=W\oplus W'$, where
$\dim W=2$, $\dim W'=1$. That is, the partition is $\pi=\{(2,1)\}$. Then $X_\pi\cong \tilde \cX_2^*$, 
since a representation $\rho\in X_\pi$ is determined by $\rho|_W$, because $\rho|_{W'}$ is
fully determined as $(\det\rho|_{W})^{-1}$. The description of the components now
follows from Proposition \ref{prop:7.3}.
\end{proof}

\begin{proposition}\label{prop:8.2}
 Then set $\cX_3^*$ of irreducible representations is composed by the following components $\cX_{3,\pi}^*$:
 \begin{itemize}
\item  $\frac1{12}(n-1)(n-2)(m-1)(m-2)$ componens of maximal dimension $4$, which are isomorphic to
$\cM/(T\x_D T)$, where $\cM\subset \GL(3,\CC)$ are the stable points for the $(T\x_D T)$-action.
\item $\frac12 (n-1)(m-1)(n+m-4)$ components $\cX_{3,\kappa}^*$, each isomorphic to $(\CC^*)^2-\{x+y=1\}$.
\end{itemize}
\end{proposition}

\begin{proof} 
The number of irreducible components of maximal dimension 
is given by  Theorem~\ref{thm:max}, and its geometric description by Proposition \ref{prop:Xtau}. 

Now we look at the remaining components. According to (\ref{eqn:kappa}), 
these are of the form $\cX_{r,\kappa}^*$, where 
$\kappa=((\epsilon_1,\epsilon_2,\epsilon_3),(\varepsilon_1,\varepsilon_2,\varepsilon_3))$ 
satisfies that some eigenvalues are repeated. From Lemma \ref{lem:red}, neither $A$ nor $B$
are a multiple of the identity for any irreducible representation $(A,B)$. Therefore, the three 
eigenvalues cannot be the same.
Also it cannot be that $\epsilon_1= \epsilon_2$ and 
$\varepsilon_1=\varepsilon_2$, since in this case the intersection of the 
two-dimensional eigenspace of $A$ with the two-dimensional eigenspace of $B$ 
would give an invariant non-trivial subspace of $\rho$. The only possibilities are:
\begin{enumerate}
 \item $\epsilon_1=\epsilon_2 \neq \epsilon_3$ and  $\varepsilon_1,\varepsilon_2,\varepsilon_3$ are distinct.
 \item $\epsilon_1,\epsilon_2,\epsilon_3$ are distinct and $\varepsilon_1=\varepsilon_2\neq \varepsilon_3$.
\end{enumerate} 

Assume that Case (1) holds. Then $\epsilon_3=\epsilon_ 1^{-2}$. As 
$\epsilon_1\neq\epsilon_3$, $\epsilon_1^3\neq 1$. On the other hand, since 
$\varpi =\epsilon^n_1$, $\epsilon^{3n}_1=1$, so there are $3 n -3 $ choices for 
$\epsilon_1$. 

Suppose $n\not\equiv 0 \pmod 3$. For each value of $\varpi$ there are $n-1$ 
choices for $\epsilon_1$. So there are a total of $(n-1)3 (m^2-3m + 
2)/6=(n-1)(m-1)(m-2)/2$ possibilities. Suppose $n \equiv 0 \pmod 3$. Then for 
$\varpi=1$ there are $n-3$ choices for $\epsilon_1$ and for $\varpi\neq 1$ there are $n$ choices. 
Note that in this case $m\not\equiv 0\pmod 3$. The total is $(3n-3) (m^2-3m + 
2)/6=(n-1)(m-1)(m-2)/2$ possibilities, again.

Now fix one strata, i.e., the eigenvalues of $A$ and $B$. Let $L,v$ be the plane 
and vector which give eigenspaces of $A$, and let $w_1,w_2,w_3$ be the 
eigenvectors of $B$. Fix coordinates so that $w_1=(1,0,0), w_2=(0,1,0), 
w_3=(0,0,1)$. The plane $L$ gives a line in $\PP^2$. It does not contain either 
of the points $[w_1],[w_2],[w_3] \in \PP^2$. Therefore the line is given by 
$a_1x_1+a_2x_2+a_3x_3=0$, that is, its coordinates in the dual space $( 
\PP^2)^*$, $[a_1,a_2,a_3]$ do not have any entry which is a zero. So we can 
rescale the coordinates to arrange $[a_1,a_2,a_3]=[1,1,1]$. Therefore we only 
have the choice of $[v]\in \PP^2$, with $[v]\notin L$. If $[v] \in R= \la 
[w_1],[w_2]\ra$ then take the intersection of $L$ with $R$, say $L\cap R=\la 
[u]\ra$. Then $\la u,v\ra$ is an invariant subspace, and so $\rho$ is reducible. 
So $[v]$ is not in any line $ \la [w_i],[w_j]\ra$. The parameter space is thus 
$\PP^2$ minus four lines: $(\CC^*)^2-\{x+y=1\}$.

Case (2) is analogous to Case (1), with $(m-1)(n-1)(n-2)/2$ strata. Note that 
$(m-1)(n-1)(n-2)/2+(n-1)(m-1)(m-2)/2=(n-1)(m-1)(n+m-4)/2$.
\end{proof}

We shall denote by $G$ the index set of those $\kappa$ parametrizing the components 
$\cX_{3,\kappa}^*$ of dimension $2$ in Proposition \ref{prop:8.2}.

\begin{theorem}  \label{thm:K-theory}
 The K-theory class of the character variety $\cX_3$ is as follows. For $m,n$ both odd,
 \begin{align*}
 [\cX_3] =& \frac1{12}(n-1)(n-2)(m-1)(m-2)( \LL^4+4\LL^3-9\LL^2-3\LL+12) \\ 
 &+ \LL^2+ \frac14 (n-1)(m-1) (\LL^2-3\LL+2) \\ &+ \frac12 (n-1)(m-1)(n+m-4) (\LL^2-3\LL+3) 
 \end{align*}
For $n$ even and $m$ odd, it is
 \begin{align*}
 [\cX_3]= &
\frac1{12}(n-1)(n-2)(m-1)(m-2)( \LL^4+4\LL^3-9\LL^2-3\LL+12)  \\ 
&+ \LL^2+\frac14 (n-2)(m-1) (\LL^2-3\LL+2) \\ &+\frac12 (m-1) (\LL^2-2\LL+1) + \frac12 (n-1)(m-1)(n+m-4) (\LL^2-3\LL+3)
 \end{align*}
\end{theorem}

\begin{proof}
We have to add the contributions from Propositions \ref{prop:8.1} and \ref{prop:8.2}. Here
$[X_{TR}]=\LL^2$, $[\cX_{3,\kappa}^*]=(\LL-1)^2-(\LL-2)=\LL^2-3\LL+3$ for $\kappa\in G$, 
$[X_{PR}^{1,i}]= (\LL-2)(\LL-1)=\LL^2-3\LL+2$ and 
$[X_{PR}^{2,i}]= (\LL-1)^2=\LL^2-2\LL+1$. It only 
remains to compute the class $[\cM/(T\x_D T)]$.

Consider the space $\cM$. We stratify it as follows:
\begin{itemize}
\item If $a_{11}, a_{21},a_{31}$ are non-zero. Then the first vector accounts for $(\LL-1)^3$. The second and third vectors
should be independent, accounting for $(\LL^3-\LL)(\LL^3-\LL^2)$. 
We have to remove the case where one of the vectors is a coordinate vector. This accounts
for $6(\LL-1)(\LL^3-\LL^2)- 6(\LL-1)^2$, the substracting term due to the overcount produced when the two vectors are
simultaneously coordinate vectors. We also need to remove when both vectors lie in a coordinate plane (and they
are not coordinate vectors, which has been already taken into account). This produces $3(\LL-1)(\LL-1)(\LL-1)(\LL-2)$.
This gives a total of 
 \begin{align*}
 (\LL-1)^3\Big((\LL^3-\LL)(\LL^3-\LL^2) &-3(\LL-1)^3(\LL-2) \\ 
 &-6(\LL-1)(\LL^3-\LL^2) + 6(\LL-1)^2 \Big).
 \end{align*}
\item Suppose one of $a_{11},a_{21},a_{31}$ is zero (say $a_{31}=0$, we have to multiply by three at the end). Then
the first vector accounts for $(\LL-1)^2$. We have a matrix 
 $$
 M=\begin{pmatrix} a_{11} & b & c \\
a_{21} & d & e \\
 0 & a_{32} & a_{33} \end{pmatrix}
 $$
It must be that $a_{32}, a_{33}$ are non-zero, so accounting for $(\LL-1)^2$. The condition for $\det(M)=0$ is linear on 
$b,c,d,e$, and all the coefficients of the linear equation are non-zero. So the choices 
for them yield a $\LL^4-\LL^3$. We have to substract for accounting the cases where 
any column or row of $\begin{pmatrix}  b & c \\
 d & e  \end{pmatrix}$  is zero, that is $4(\LL^2-\LL)-4(\LL-1)$. So the total is
 $$
 3(\LL-1)^4 \left( \LL^4-\LL^3- 4(\LL^2-\LL) + 4(\LL-1)\right).
$$
\end{itemize}
Therefore 
 \begin{align*}
 [\cM/T\x_D T] &= \frac{[\cM]}{(\LL-1)^5} = \frac{1}{(\LL-1)^2} \Big((\LL^3-\LL)(\LL^3-\LL^2)  
- 3(\LL-1)^3(\LL-2) \\ & \quad  -6(\LL-1)(\LL^3-\LL^2) + 6(\LL-1)^2 \Big)
 + \frac{3(\LL^4-\LL^3- 4(\LL^2-\LL) + 4(\LL-1) )}{\LL-1} \\ 
 &= 
  \LL^4+4\LL^3-9\LL^2-3\LL+12. 
 \end{align*} 
\end{proof}

\begin{corollary} \label{cor:n,m}
The character variety $\cX_3$ determines $n,m$ up to order.
\end{corollary}

 \begin{proof}
With $\cX_3$ we have the class $[\cX_3]\in K_0(\VarC)$ given in Theorem \ref{thm:K-theory}. The coefficient of 
$\LL^4$ gives us $(n-1)(n-2)(m-1)(m-2)$. Now we substract the 
term $\frac1{12}(n-1)(n-2)(m-1)(m-2)(    \LL^4+4\LL^3-9\LL^2-3\LL+12)$. 
In the expression that we obtain, we substitute $\LL \to 0$, producing $p=\frac12 (n-1)(m-1) (3n+3m-11)$.
Substituting $\LL \to 1$, it yields $q=1+\frac12 (n-1)(m-1)(n+m-4)$. The quantity 
$2p-6q+6=(n-1)(m-1)$. Now we can recover $(n-2)(m-2)$ as well, and with this
we get $nm$ and $n+m$. This proves the result.
\end{proof}

\begin{remark}\label{rem:otro-mas}
Note also that we can describe the $4$-dimensional components $\cM/(T\x_D T)$ of Proposition \ref{prop:8.2} as follows.
Each column of the matrix (\ref{eqn:M}) gives a point $p_j=[a_{1j},a_{2j},a_{3j}]\in \PP^2$. These points must 
be independent, they cannot be equal to $q_1=[1,0,0],q_2= [0,1,0]$ or $q_3=[0,0,1]$, and two of them cannot
be simultaneously in a coordinate line $L_1=\la q_2,q_3\ra,  L_2=\la q_1,q_3\ra$ or  $L_3=\la q_1,q_2\ra$.
Let 
 \begin{align*}
 \cV = \{( &p_1,p_2,p_3) \in (\PP^2-\{q_1,q_2,q_3\})^3 \text{ independent} | \\ 
 & \text{no two of them lie simultaneously in
either } L_1, L_2,L_3\}
\end{align*}
Consider the $(\CC^*)^2$ action given by $p_j=[a_{1j},a_{2j},a_{3j}] \mapsto [\alpha a_{1j},\beta a_{2j},a_{3j}]$, $j=1,2,3$.
Then 
 $$
 \cM/(T\x_D T)=\cV/(\CC^*)^2. 
 $$
\end{remark}

\section{Intersection patterns} \label{sec:inters}

We have the stratification
 \begin{equation}\label{eqn:strat}
\cX_3=X_{TR}\sqcup \left( \bigsqcup X_{PR}^{1,i} \right) \sqcup \left(\bigsqcup X_{PR}^{2,j}\right)  \sqcup
 \left( \bigsqcup_{\tau \in F} \cX_{3,\tau}^{*} \right)
\sqcup \left( \bigsqcup_{\kappa\in G} \cX_{3,\kappa}^{*} \right) 
 \end{equation}
into totally reducible representations, partially reducible (of type 1 and type 2), 
and irreducible representations (of type I and of type II).
The set $X_{TR}$ is closed.
We want to describe how the closures $\overline{X_{PR}^{1,i}}$, $\overline{X_{PR}^{2,j}}$, $\overline{\cX_{3,\kappa}^*}$
and $\overline{\cX_{3,\tau}^*}$ intersect the other strata.

The component $X_{TR}\cong \CC^2$ is closed. By Proposition \ref{prop:3}, it is parametrised by
$(x,y)$, where $x=t_1+t_2+t_3$, $y=t_1t_2+t_1t_3+t_2t_3$,
$t_1t_2t_3=1$, the matrices $(A,B)$ being 
$A=\diag(t_1^m,t_2^m,t_3^m)$, $B=\diag(t_1^n,t_2^n,t_3^n)$.
  
\begin{proposition}
The partially reducible components of both types, $X_{PR}^{1,i} = \CC^*\x 
(\CC-\{0,1\})$ and $X_{PR}^{2,j} = \CC^*\x \CC^*$,  have closure $\CC^*\x 
\CC$. Their intersection with $X_{TR}$ are the curves with equation
$$
x^2 y^2-(c_k+2)(x^3+y^3)+(c_k^2+5 c_k+4) x y -(c_k+1)^3=0
$$
where $c_k=2\cos \frac{2\pi k}{m n}$ for $k\in\mathbb Z$, $k\not\in m 
\mathbb{Z}$, and $k\not\in n \mathbb{Z}$.

Two curves indexed by $k$ and $k'$ belong to the closure of same component if and only if
$k'\equiv \pm k \pmod m$ and   $k'\equiv \pm k \pmod n$. A curve belongs the 
closure of a type 2 component precisely when, with $n$ even, $k\equiv n/2 \pmod n$.
\end{proposition}

\begin{proof}
The first assertion follows from the description of the closure of irreducible 
components in $\cX_2$, and from the relationship between $\cX_2$, $\bar \cX_2$ 
and $\tilde \cX_2$, as $\tilde \cX_2^*$ is isomorphic to the locus of partially 
reducible representations of $\cX_3$.

To describe the incidence of those components, we first discuss the incidence 
in $\cX_2$ of reducible and irreducible components. According to \cite{HPS}, the 
diagonal representation $A=\diag( s^m, s^{-m})$, $B=\diag (s^n, s^{-n})$ can be 
deformed into an irreducible representation
in $\SL(2,\CC)$  iff $\Delta_{m,n}(s^2)=0$, where 
$$
\Delta_{m,n}=\frac{(t^{m n}-1)(t-1)}{(t^m-1)(t^n-1)}
$$
is the Alexander polynomial of the torus knot.  For matrices of $\GL(3,\CC)$, 
this leads to curves in $\overline{X_{PR}^{\bullet,i}}\cap X_{TR}$, determined 
by matrices $\diag(\delta\,t,\delta^{-1}t, t^{-2})$, with 
$\Delta_{m.n}(\delta^2)=0$. In the coordinates of $\cX_{3,TR}$, the equations 
with a parameter $t\in\mathbb{C}^*$ are
$$
x=(\delta+\delta^{-1}) t+t^{-2}, \qquad 
y=( \delta+\delta^{-1} )  t^{-1}+t^{2},  
$$
with $\Delta_{m.n}(\delta^2)=0 $. Setting $c=\delta^2+\delta^{-2}$, we obtain 
the curves 
of the statement. This gives $(m-1)(n-1)/2$ curves, hence we obtain  
all of them.

The assertion on the components is proved by using the argument of 
Proposition~\ref{prop:cX2}. Namely the components of $\cX_2$ are determined by 
the eigenvalues of $A$ and $B$, that are $\{\delta^m,\delta^{-m}\}$ and   
$\{\delta^n,\delta^{-n}\}$ respectively, and then it follows from the discussion 
for the components of $\tilde \cX_2^*$ and the partially reducible components of 
$\cX_3$.
\end{proof}

\begin{proposition}
The irreducible components of type I, $\cX_{3,\tau}^*= \cM/H$, $H=T\x_D T$, have
closure $\overline{\cX_{3,\tau}^*}=\GL(3,\CC)//H$. 
The boundary strata are given by 
 \begin{itemize}
 \item Orbits coming from $3\x 3$-matrices $M$ with two zeroes in a row or a column. The representation
is S-equivalent to $V\oplus W$, so it lies in some $X_{PR}^{\bullet,i}$.
 \item Orbits coming from $3\x 3$-matrices  $M$ with three zeroes, two in a row, and two in a 
column (one of these zeroes in common). This is S-equivalent to a diagonal matrix, i.e., a 
totally reducible representation $V\oplus V'\oplus V''$ lying in $X_{TR}$.
\end{itemize}
The boundary strata consist of\/ $9$ lines of partially reducible representations 
that intersect each other following the pattern of the full graph $K_{3,3}$.
In addition the intersection points are precisely the totally reducible 
representations.
\end{proposition}

\begin{proof}
The fact that the closure is $ \GL(3,\CC)//H$ is explained in  
Remark~\ref{rem:x}. The choice of two or three zeroes as in the statement
determines an invariant subspace or an invariant flag of the representation. Furthermore
in the closure of the orbit of this representation we find another representation
that is direct sum of irreducible ones (a semisimple one). This corresponds 
to the partially reducible and the totally reducible representations, and the
eigenvalues of each factor can be computed. In particular, for the partially reducible 
representations, we obtain a whole component of $\cX^*_2$ isomorphic
to $\CC-\{0,1\}$, whose closure is $\CC$, by adding totally reducible 
representations.

The pattern of the compactification locus is the following: 
\begin{itemize}
\item Draw a point for each totally reducible representation. This is the same 
as
selecting three entries in the matrix not in the same row or column. There is a 
total of $6$.
\item Draw a line for each partially reducible representation. This is the same 
as fixing one entry in
the matrix. The $2$-dimensional representation is given by
the $2\x 2$-minor associated to it. There is a total of 
$9$. 
\item Every line contains two points. Every point is in three lines.
\item The pattern is the full graph $K_{3,3}$. It consists of all the
edges connecting $3$ points at the top with $3$ points 
at the bottom in all possible ways.
\end{itemize}
\end{proof}

\begin{proposition} The irreducible components of type II, $\cX_{3,\kappa}^*
=\CC^2-(\{x=0\}\cup \{y=0\}\cup  \{x+y=1\})$  have closure 
$\overline{\cX_{3,\kappa}^*}=\CC^2$.  The closure consists of adding three 
lines of partially reducible representations that intersect pairwise. The three 
intersection points are precisely the totally reducible representations. 
\end{proposition}

\begin{proof} Recall that in this component $A$ has an eigenvalue with 
multiplicity two and the eigenvalues of $B$ are different (it could be the same changing
the roles of $A,B$; the analysis would be similar in that case). This gives a point 
$p_0\in \PP^2$ and a line $l_0\subset \PP^2$ fixed by $A$, and three points 
$p_1,p_2,p_3\in \PP^2$ fixed by $B$, all points and lines in generic position. 
Putting $l_0$ as line at infinity, $p_1,p_2,p_3$ define an affine frame of 
$\CC^2$ and the position of $p_0$ parametrizes the component. The closure 
is obtained by allowing $p_0$ to be at any position in $\CC^2$. Notice that if 
$p_0$ belongs to the line through $p_1$ and $p_2$, then, since this line also 
meets $l_0$, it is preserved by both $A$ and $B$, hence we obtain  a partially 
reducible representation. Total reducibility is obtained when $p_0$ equals one 
of $p_1,p_2,p_3\in \PP^2$, because there are two invariant projective lines. 
\end{proof}

Next we want to understand the intersection of  the components of $X_{PR}^{\bullet,i}$ with 
the closure of
the components of irreducible representations. A component $X_{PR}^{\bullet,i}$
is determined by eigenvalues $\{\epsilon,1/\epsilon\}$
and $\{\varepsilon,1/\varepsilon\}$ corresponding to an irreducible
representation in  $\SL(2,\CC)$. The twisted Alexander polynomial
for the representation in $\SL(2,\CC)$ is constant on each component and it 
equals
$$
\Delta_{\epsilon,\varepsilon}(t)=\frac{(t^{mn}-\epsilon^n)^2}{ 
(t^m-\epsilon)(t^m-1/\epsilon)(t^n-\varepsilon)(t^n-1/\varepsilon)}.
$$
See \cite{KM} for the computation of $
\Delta_{\epsilon,\varepsilon}(t)$,  or \cite{HeusenerP} for the trefoil.

Recall that $X_{PR}^{1,i}\cong (\CC-\{0,1\})\times \CC^*$ and that 
$X_{PR}^{2,i}\cong ((\CC-\{0,1\})\times \CC^*)/\mu_2$. Let $z\in\CC^*$ denote the 
coordinate in the factor $\CC^*$.

\begin{proposition}
 A component $X_{PR}^{\bullet,i}$ intersects the closure $\cX^*_3$ precisely at the 
curves $\{z=z_0\}$ where $\Delta_{\epsilon,\varepsilon}(z_0^3)=0$.
 In addition,
a curve defined by $z=z_0$ lies in the closure of a four  dimensional component 
if $z^3_0$ is a root of multiplicity two of $\Delta_{\epsilon,\varepsilon}$, 
and a two dimensional component if it is a simple root. 
\end{proposition}

\begin{proof}
By \cite{HeusenerP}, $\Delta _{\epsilon,\varepsilon}(z^3)=0$ is a necessary 
condition for a representation with second coordinate $z\in\CC^*$ to be deformed 
to irreducible representations. In addition, for simple roots this condition is 
also sufficient and the component of the character variety of irreducible ones 
has dimension two.  Thus we prove the lemma by counting the number of such 
curves obtained from the roots, and the number of curves in the closure of the 
variety of irreducible characters. Firstly, the degree of each polynomial 
$\Delta_{\epsilon,\varepsilon}(t^3)$ is $3(2 m\, n -2m-2n)$. For $m$ and $n$ odd 
there are $\frac{m-1}2\frac{n-1}2$ possibilities for the eigenvalues 
$\{\epsilon,1/\epsilon\}$ and $\{\varepsilon,1/\varepsilon\}$, thus we find at 
most $ 3 \frac{m-1}2\frac{n-1}2 2 (m\, n-m-n)$ curves in the closure, counted 
twice for double roots. When $n$ is even, we have $\frac{m-1}2\frac{n-2}2$ 
components as above but also $(m-1)/2$ components  that are the quotient of a 
component of $\cX_2\times \CC^*$ by $\mu_2$. The contribution of those 
components is half the contribution of the other components, thus we get the 
same upper bound: $ 3( \frac{m-1}2\frac{n-2}2 + \frac{m-1}4) 2 (m\, n-m-n)= 3 
\frac{m-1}2\frac{n-1}2 2 (m\, n-m-n)$. 

On the other hand, there are $\frac 1{12}(n-1)(n-2)(m-1)(m-2)$ components of 
irreducible representation of dimension 4 that have $9$ lines in the adherence 
and that 
contribute to double roots. There are also $\frac 1{2}(n-1)(m-2)(m+n-4)$
components of dimension two that 
contribute with three lines each. Thus the total number of such lines (counted 
twice in the closure of a four dimensional component) is:
\begin{multline*}
2\cdot 9\frac 1{12}(n-1)(n-2)(m-1)(m-2) + 3 \frac 1{2}(n-1)(m-2)(m+n-4)
\\
= \frac 32 
(n-1)(m-1)(m\, n -m-n). 
\end{multline*}
As we get precisely the previous upper bound,  the proposition follows. 
\end{proof}

\section{Character varieties for $\GL(3,\CC)$ and $\PGL(3,\CC)$} \label{sec:GL3}

Now we describe the $\GL(3,\CC)$ and $\PGL(3,\CC)$-character varieties $\tilde\cX_3$ and $\bar\cX_3$.

\begin{proposition}
\label{prop:psl3} The components of the $\PGL(3,\CC)$-character variety  $\bar\cX_3$ are:
  \begin{itemize}
   \item The component of totally reducible representations, which is  isomorphic to  $\CC^2/\mu_3\cong \{(x,y,z)\in \CC^3\mid
   x\, y= z^3\}$.
   \item   $[\frac{n-1}2][\frac{m-1}2]$ components of partially reducible representations, each 
   isomorphic to $(\mathbb C -\{0,1\})\times \mathbb{C}^*$.
   \item When $n$ is even, there are $(m-1)/2$ additional components of  partially reducible representations, each 
   isomorphic to $\{(u,v)\in \CC^2 | v\neq 0, v\neq u^2\}$.
   \item When $m, n\not \in 3\ZZ $, there are the following components of irreducible representations:
   \begin{itemize}
    \item $(n-1)(m-1)(n+m-4)/6$ components isomorphic to $ (\CC^*)^2 -\{x+y=1\}$
    \item and $(m-1)(m-2)(n-1)(n-2)/36$ components of maximal dimension 
isomorphic to $\cM/(T\times_D T)$.
   \end{itemize}
    \item When $n \in 3\ZZ$, there are the following components of irreducible representations:
    \begin{itemize}
	\item   $(m-1)(m\, n + n^2- 5 n- m+2)/6$ components isomorphic to $ (\CC^*)^2 -\{x+y=1\}$,
	\item  $m-1$ components isomorphic to  $
	\{(x,y,z)\in\CC^3\mid x\, y= z^3,\ x+ y + 3z\neq 1\}$,
	\item  $(m-1)(m-2)n(n-3)/36$  components of maximal dimension  isomorphic to $\cM/(T\times_D T)$,
	\item and $(m-1)(m-2)/6$  components of maximal dimension  isomorphic to $\cM/(T\times_D T\rtimes \mu_3)$, where
	$\mu_3$ acts by cyclic permutation of columns in $\cM$.
    \end{itemize}
    The case $m\in 3\ZZ$ is symmetric.
\end{itemize}
\end{proposition}

\begin{proof}
Use the isomorphism  $\bar\cX_3 = \cX_3/\mu_3$ by \eqref{eqn:2}, where $\mu_3=\{1,\varpi,\varpi^2\}$,
and the stratification (\ref{eqn:strat}).
 The component $X_{TR}\cong\mathbb C^2$ is invariant by $\mu_3$ and  $\varpi$ 
 maps a point with coordinates $(x,y)$ to $(\varpi x, 
\varpi^2 y)$. Hence the quotient is $\CC^2/\mu_3 \cong \{(u,v,w)\in \CC^3 | w^3=u\,v\}$, by $u=x^3, v=y^3, w=x\, y$.

The components $X_{PR}^{1,i}$ correspond to components of 
$\tilde \cX_2^*\times \CC^*$ that are not preserved by $\mu_2$. Here 
$\varpi$ acts trivially on  $\tilde \cX_2^*$ and by multiplication by a third 
rood of unity on $ \mathbb{C}^*$. This yields $[\frac{n-1}2][\frac{m-1}2]$ components 
isomorphic to $((\mathbb C -\{0,1\})\times \mathbb{C}^*)/\mu_3\cong (\mathbb C -\{0,1\})\times \mathbb{C}^*$. 

For $n$ even, the 
components $X_{PR}^{2,j}$ are isomorphic to $\{(u,v)\in \CC^2 | v\neq 0, v\neq 
u^3\}$. To understand the action of $\mu_3$, we build on the proof of \ref{prop:7.3}.
Recall from the proof of Proposition~\ref{prop:7.3} that $u=(2 r-1)\lambda$ and
$v=\lambda^2$, where $r\in (\CC-\{0,1\}$ is the coordinate of a component 
$\cX_2$ and $\lambda\in \CC^*$. Since $\varpi$ acts trivially on $r$ and maps
$\lambda$ to $\varpi\lambda$, the action of $\varpi$ on those coordinates is 
$(u,v)\mapsto (\varpi u, \varpi^2 v)$. The 
quotient is isomorphic to 
$\{(z,w)\in \CC^2 | w\neq 0, w\neq z^2\}$, via $z=u/v^2, w=1/v^3$.
 
We next discuss the irreducible components. The two-dimensional components 
$\cX_{3,\kappa}^{*}$, $\kappa \in G$, are parametrized by the eigenvalues of $A$ and $B$: 
$\{\epsilon,\epsilon,1/\epsilon^2\}$ and 
$\{\varepsilon_1,\varepsilon_2,\varepsilon_3\}$ subject to 
$\epsilon^n=\varepsilon_i^m=\varpi^k$, $\varepsilon_1 
\varepsilon_2 \varepsilon_3=1$, $\epsilon\neq1/\epsilon^2$ and 
$\varepsilon_1,\varepsilon_2,\varepsilon_3$ distinct. $\varpi$ maps those eigenvalues to 
$\{\varpi^m\epsilon,\varpi^m\epsilon,1/(\varpi^m\epsilon)^2\}$ and 
$\{\varepsilon_1 \varpi^n,\varepsilon_2\varpi^n,\varepsilon_3 \varpi^n\}$. They 
happen to be the same set of eigenvalues precisely when $m\in 3\ZZ$ and 
$\{\varepsilon_1,\varepsilon_2,\varepsilon_3\}=\{1,\varpi,\varpi^2\}$. Therefore 
$\mu_3$ permutes all components $\cX_{3,\kappa}^{*}$ except when  $m\in 3\ZZ$, 
that $n-1$ components are preserved. Thus when $m, n\not\in 3\ZZ$ we obtain 
$(n-1)(m-1)(n+m-4)/6$ components isomorphic to $ (\CC^*)^2 - \{x+y=1\}$. When 
$m\in 3\ZZ$ we obtain $(n-1)(m\, n+m^2- 5 m-n+2)/6$ such components, and  $n-1$ 
additional components isomorphic to $( (\CC^*)^2 - \{x+y=1\})/\mu_3$, where 
$\mu_3$ acts linearly on $\CC^2$ and cyclically permutes the three lines that we have removed, 
that is $(x,y)\mapsto (1-x-y,x)$. 
The quotient is isomorphic to $\{(u,v,w)\in\CC^3\mid u\, v= w^3 , \  u+ v +3w\neq 1\}$, 
by taking coordinates $u=\frac19 ((3x-1)-\varpi (3y-1))^3, 
v=\frac19((3x-1)-\varpi^2 (3y-1))^3, w=9^{-2/3} ((3x-1)-\varpi (3y-1))((3x-1)-\varpi^2 (3y-1))$.
 
The four dimensional components  $\cX_{3,\kappa}^{*}$  are parametrized by  
$\{\epsilon_1,\epsilon_2,\epsilon_3\}$ and 
$\{\varepsilon_1,\varepsilon_2,\varepsilon_3\}$ distinct, subject to 
$\epsilon_i^n=\varepsilon_j^m=\varpi^k$, 
$\epsilon_1\epsilon_2\epsilon_3=\varepsilon_1 \varepsilon_2 \varepsilon_3=1$,
being the eigenvalues of $A$ and $B$ 
respectively. The generator of the cyclic group maps those eigenvalues to 
$\{\epsilon_1 \varpi^m,\epsilon_2\varpi^m,\epsilon_3 \varpi^m\}$ and  
$\{\varepsilon_1 \varpi^n,\varepsilon_2\varpi^n,\varepsilon_3 \varpi^n\}$. An 
elementary computation proves that  a component is invariant precisely when  
$n\in 3\ZZ$ and 
$\{\epsilon_1,\epsilon_2,\epsilon_3\}=\{1,\varpi,\varpi^2\}$, or $m\in 3\ZZ$ 
and $\{\varepsilon_1,\varepsilon_2,\varepsilon_3\}=\{1,\varpi,\varpi^2\}$. 
Hence when $m,n\not \in 3\ZZ$ we obtain $(m-1)(m-2)(n-1)(n-2)/36$ components 
isomorphic to the components of $\cX_{3,\kappa}^{*}$. When $n \in 3\ZZ$ we 
obtain $(m-1)(m-2)n(n-3)/36$ such components and $(m-1)(m-2)/6$ components that 
are quotiented by the action of $\mu_3$, which can be interpreted as cyclic 
permutation of columns in $\cM$.
\end{proof}
  
\begin{remark}
The closure of each component can be easily deduced from the closure of 
the components of $\cX_3$. Just notice that, when $n\in 3\ZZ$,
the closure of each component isomorphic to
$\{(x,y,z)\in\CC^3\mid x\, y= z^3,\ x+ y + 3z\neq 1\}$ is
the hypersurface $x\, y= z^3$ and 
the curve of reducible 
representations $x+ y + 3z = 1$ is singular exactly at $-x=-y=z=1$, the  
totally reducible representation. 
\end{remark}

\begin{corollary}
 The K-theory class of the character variety $\bar\cX_3$ is as follows:
\begin{itemize}
\item If $n,m\equiv 1,5 \pmod 6$, then $[\bar\cX_3]=P_0+
   \frac1{36} (m-1)(m-2)(n-1)(n-2) P_1+
   \frac16(n-1)(m-1)(n+m-4) P_3 + 
   \frac14(n-1)(m-1)  P_5$.
\item If $n\equiv 2,4 \pmod 6$, $m\equiv 1,5\pmod 6$, then $[\bar\cX_3]= P_0+
   \frac1{36} (m-1)(m-2)(n-1)(n-2) P_1+
   \frac16(n-1)(m-1)(n+m-4) P_3 + 
   \frac14(n-2)(m-1)  P_5 + \frac12 (m-1) P_6$.
\item If $n\equiv 3 \pmod 6$,  $m\equiv 1,5\pmod 6$, then $[\bar\cX_3]= P_0+
   \frac1{36} (m-1)(m-2)n(n-3) P_1+ \frac16(m-1)(m-2) P_2+
   \frac16(m-1)(mn+n^2-5n-m-2) P_3 + (m-1) P_4+
   \frac14(n-1)(m-1)  P_5$.
\item If $n\equiv 0 \pmod 6$,  $m\equiv 1,5\pmod 6$, then $[\bar\cX_3]= P_0+
   \frac1{36} (m-1)(m-2)n(n-3) P_1+ \frac16(m-1)(m-2) P_2+
   \frac16(m-1)(mn+n^2-5n-m-2) P_3 + (m-1) P_4+
   \frac14(n-2)(m-1)  P_5 + \frac12 (m-1) P_6$.
\item If $n\equiv 2,4 \pmod 6$, $m\equiv 3 \pmod 6$, then $[\bar\cX_3]= P_0+
   \frac1{36}m (m-3) (n-1)(n-2) P_1+ \frac16(n-1)(n-2) P_2+
   \frac16(n-1)(mn+m^2-n-5m-2) P_3 + (n-1) P_4+
   \frac14(n-2)(m-1)  P_5 + \frac12 (m-1) P_6$.
\end{itemize}
Here  $P_0=\LL^2$, $P_1=   \LL^4+4\LL^3-9\LL^2-3\LL+12$, 
$P_2= \LL^4 +2\LL^3-3\LL^2- \LL+4$, $P_3=\LL^2-3\LL+3$,
$P_4=\LL^2-\LL+1$, $P_5=\LL^2-3\LL+2$, $P_6=\LL^2-2\LL+1$.
\end{corollary}

\begin{proof}
The proof is analogous to that of Theorem \ref{thm:K-theory}. 
First $[\CC^2/\mu_3]=\LL^2$ by Proposition \ref{prop:quotient}. Second,
$[( (\CC^*)^2 - \{x+y=1\})/\mu_3]=\LL^2-(\LL-1)$, since the three lines
in the quotient produce a single line with two points identified.
Finally, it remains to compute the
K-theory class $[\cM/(T\times_D T\rtimes \mu_3)]$.

For this, we use the description of $\cM/(T\times_D T)$ given in Remark \ref{rem:otro-mas}. The
action of $\mu_3$ is given by cyclic permutation of $p_1,p_2,p_3$. We divide the computation in
two cases:
 \begin{itemize}
\item All $p_1,p_2,p_3$ lie off $L_1$. Then $p_1,p_2,p_3\in \CC^2=\PP^2-L_1$. The
quotient $((\CC^2)^3)/\mu_3$ has class $\LL^6$, by Proposition \ref{prop:quotient}. Now we
have to remove various contributions as given in the definition of $\cV$:
 \begin{enumerate}
 \item If $p_i$ are equal, we get $\LL^2$.
 \item $p_i$ are not equal but lie on a line. The space parametrizing
lines is a $\PP^2$ minus a point corresponding to the line $L_1$, hence giving $[\PP^2]-1=\LL^2+\LL$.
The space parametrizing triples of points in a line is $[(\CC^3)/\mu_3]=\LL^3$, by Proposition \ref{prop:quotient}.
As we are considering unequal points, we have $\LL^3-\LL$. Multiplying, we get $(\LL^2+\LL)(\LL^3-\LL)$.
\item $p_i$ are independent and one is equal to $q_1$. Using $\mu_3$, we can suppose $p_1=q_1$. Then
$p_2\in \CC^2-\{q_1\}$, and $p_3$ lies off the line $\la q_1,p_2\ra$. This yields $(\LL^2-1)(\LL^2-\LL)$.
\item Two of $p_1,p_2,p_3$ lie on $L_2$ or $L_3$ (and the three are independent and not $q_1$). Using
the cyclic permutation, we can assume this happens to $p_1,p_2$. Then we have $2(\LL-1)(\LL-2)(\LL^2-\LL)$,
the last factor accounting for the fact that $p_3\notin \la p_1,p_2\ra$.
\end{enumerate}
This gives a total of 
 $\LL^6-  (\LL^2+(\LL^2+\LL)(\LL^3-\LL)+ (\LL^2-1)(\LL^2-\LL)+2(\LL-1)(\LL-2)(\LL^2-\LL) =
( \LL-1)^2 (\LL^4 +\LL^3-3\LL^2+3\LL)$.
\item If one of  $p_1,p_2,p_3$ lie in $L_1$, the cyclic permutation allows to assume that $p_1\in L_1-\{q_2,q_3\}$. 
This produces a factor $(\LL-1)$.
Also $p_2,p_3\in \CC^2$, producing $\LL^4$. We remove:
 \begin{enumerate}
 \item If $p_i$ are equal, we get $\LL^2$.
 \item $p_i$ are not equal but they are collinear. This gives $\LL^2(\LL-1)$.
\item One $p_i$ is equal to $q_1$, and the remaining $p_j$ is not collinear with the other two.
This yields $2 (\LL^2-\LL)$.
\item $p_2,p_3$ lie on $L_2$ or $L_3$ (and are not $q_1$). Then we have $2(\LL-1)(\LL-2)$.
\end{enumerate}
This gives a total of 
$(\LL-1) ( \LL^4-  (\LL^2+\LL(\LL^2-\LL)+2 (\LL^2-\LL)+2(\LL-1)(\LL-2))) =( \LL-1)^2 (\LL^3-4\LL+4 )$.
\end{itemize}
Adding both contributions and dividing by $(\LL-1)^2$, we get
 $$
[\cM/(T\times_D T\rtimes \mu_3)]= \LL^4 +2\LL^3-3\LL^2- \LL+4 .
 $$
\end{proof}

An argument similar (but longer) to that in Corollary \ref{cor:n,m} proves that 
one can recover $n,m$ up to order from the K-theory class $[\bar\cX_3]$.

\begin{proposition} The components of the $\GL(3,\CC)$-character variety  
$\tilde\cX_3$ are:
  \begin{itemize}
   \item The component of totally reducible representations, isomorphic to $\CC^2\times \CC^*$.
   \item  $[\frac{n-1}2][\frac{m-1}2]$  components of partially reducible 
representations, each 
   isomorphic to $(\mathbb C -\{0,1\})\times (\mathbb{C}^*)^2$.
   \item When $n$ is even, there are $(m-1)/2$ additional components of  
partially reducible representations, each 
   isomorphic to $\{(x,y,z)\in \CC^3 | y,z\neq 0, y\neq z^2\}$.
   \item When $m, n\not \in3\ZZ$, there are the following components of 
irreducible representations:
   \begin{itemize}
    \item $(n-1)(m-1)(n+m-4)/6$ components isomorphic to $ (\CC^*)^2 -\{x+y=1\}$
    \item and $(m-1)(m-2)(n-1)(n-2)/36$ components of maximal dimension 
isomorphic to $\cM/(T\times_D T)$.
   \end{itemize}
    \item When $n\in 3\ZZ $, there are the following components of irreducible 
representations:
    \begin{itemize}
	\item   $(m-1)(m\, n + n^2- 5 n- m+2)/6$ components isomorphic to 
 $((\CC^*)^2 -\{x+y=1\})\times \CC^*$.
	\item  $m-1$ components isomorphic to  $\{(u,v,w)\in\CC^3\mid u^3+v^3+3 
uv -w\neq 0, w\neq 0\}$.
	\item  $(m-1)(m-2)n(n-3)/36$  components of maximal dimension  
isomorphic to $\cM/(T\times_D T)\times \CC^*$.	
	\item $(m-1)(m-2)/6$  components of maximal dimension  isomorphic to 
$(\cM/(T\times_D T)\times \CC^*)/\mu_3$, where
	$\mu_3$ acts by cyclic permutation of columns in $\cM$ and 
multiplication on $\CC^*$.
    \end{itemize}
    The case $m\in 3\ZZ $ is symmetric.
\end{itemize}
\end{proposition}
  
To prove this proposition, one may use the  isomorphism  between $\tilde\cX_3$ 
and $(\cX_3\times \CC^*)/\mu_3$  and the proof  is analogous to the discussion 
in the proof of Proposition~\ref{prop:psl3}. Namely, the components of $\cX_3$ 
that are invariant (or not) by the action of $\mu_3$ correspond to the 
components of $\cX_3\times \CC^*$ that are invariant (or not), and the 
computation of the quotients of the invariant ones is changed by the factor 
$\CC^*$.

Finally, the K-theory class is $[\tilde\cX_3]=(\LL-1)[\bar\cX_3]$.

\end{document}